\documentclass[12pt]{amsart}
\usepackage{amssymb}
\usepackage{amsmath,amssymb,amsfonts,amsthm,latexsym}
\usepackage{mathrsfs}
\usepackage{color}
\usepackage{eucal}
\usepackage{xspace}
\usepackage{hyperref}

\usepackage{graphicx}
\usepackage{bm}

\theoremstyle{plain}
\newtheorem{theorem}{Theorem}

\newtheorem{proposition}{Proposition}

\theoremstyle{definition}

\newtheorem{remark}{Remark}
\theoremstyle{remark}

\numberwithin{equation}{section}

\usepackage[round]{natbib}
\begin{document}
	\begin{center}
	  {\bf\Large  RNA secondary structures having a compatible sequence of certain nucleotide ratios
		}
		\\
		\vspace{15pt}  Christopher L. Barrett, Thomas J. X. Li and Christian M. Reidys$^{\,\star}$
	\end{center}
	
	\begin{center}
		Biocomplexity Institute of Virginia Tech\\
		Blacksburg, VA 24061, USA\\
		Email: cbarrett@vbi.vt.edu, thomasli@vbi.vt.edu,  duckcr@vbi.vt.edu
	\end{center}
	
	\centerline{\bf Abstract}{\small
          Given a random RNA secondary structure, $S$, we study RNA sequences having fixed ratios of nuclotides
          that are compatible with $S$. We perform this analysis 
          for RNA secondary structures subject to various base pairing rules and minimum arc- and
          stack-length restrictions.
          Our main result reads as follows: in the simplex of the nucleotide ratios there exists a
          convex region in which, in the limit of long sequences, a random structure a.a.s.~has
          compatible sequence with these ratios and outside of which a.a.s.~a random structure has no such compatible
          sequence.
          We localize this region for RNA secondary structures subject to various base pairing rules and
          minimum arc- and stack-length restrictions.
          In particular, for {\bf GC}-sequences having a ratio of {\bf G} nucleotides smaller than $1/3$,
          a random RNA secondary structure without any minimum arc- and stack-length restrictions has
          a.a.s.~no such compatible sequence. For sequences having a ratio of {\bf G} nucleotides larger than
          $1/3$, a random RNA secondary structure has a.a.s. such compatible sequences.
          We discuss our results in the context of various families of RNA structures.
        }

		
		




\section{Introduction}


This paper is motivated by the question whether there is ``more'' to DNA sequence data than the mere sequence
of nucleotides, identified by alignment methods. As DNA is transcribed into RNA and subsequently processed in
various ways, we ask here the question what effect the frequencies of the nucleotides have on the variety
of RNA structures that are compatible with the given sequence. Sequences are called compatible if they satisfy
the base pairing rules for all bonds of the structure.
In particular, for RNA secondary structures the
work of McCaskill~\citep{McCaskill}, allows to efficiently Boltzmann-sample such configurations.
Simple quantities such as the nucleotide ratios could be used to identify the functionality of embedded DNA sequences.

This paper provides the mathematical analysis of the fact that in the simplex of the nucleotide ratios there exists
a convex region in which, in the limit of long sequences, a random structure a.a.s.~has compatible sequences with these
ratios and outside of which a.a.s.~a random structure has no such compatible sequence.
Here a.a.s.~stands for asymptotically almost surely, i.e.~in the limit of long sequences, the probability tends to 1.
The nucleotide ratios are a discriminant of whether or not a random secondary structure has such a compatible
sequence and is thus accessible by means of folding. By random we mean a secondary structure that, in the limit of long
sequences, has the average number of arcs, see Theorem~\ref{T:arcclt} for details.
On the flip side, given a biologically relevant sequence, whose ratios are not within this window, we can test for
whether or not functionalities are achieved by other means, i.e.~for instance by interactions with either Proteins
or other RNA structures.

Before we proceed by providing some background on RNA secondary structures as contact structures and diagrams, we
validate our modeling Ansatz. As we shall show in Section~\ref{S:PT}, the key lies in a central limit theorem
for the number of arcs in biologically ``realistic'' RNA secondary structures. Here realistic is tantamount to RNA secondary
structures having minimal free energy. These structures do not contain bonds that force an extreme bending of the
backbone and they contain base pairs organized in stacks of length at least three. The latter reflects the fact that the main
contribution of lowering free energy stems from electron delocalization between stacked bonds~\citep{Hunter:90,Sponer:01,Sponer:13}.

The space of minimum free energy (mfe) RNA secondary structures can only be understood by
means of extensive computation. The space of {\it all} RNA secondary structures that satisfy the above mentioned
constraints, however, can be mathematically analyzed.
It comes down to the question of whether or not the number of arcs of mfe-structures can be obtained by uniformly sampling
RNA secondary structures satisfying minimum stack-length three and arc-length four constraints. In Figure~\ref{F:result} we contrast these two distributions.
The similarity of the two distributions allows us to pass to combinatorics and thus to prove statements in the
larger space.
\begin{figure}
	\centering
	\includegraphics[width=0.8\textwidth]{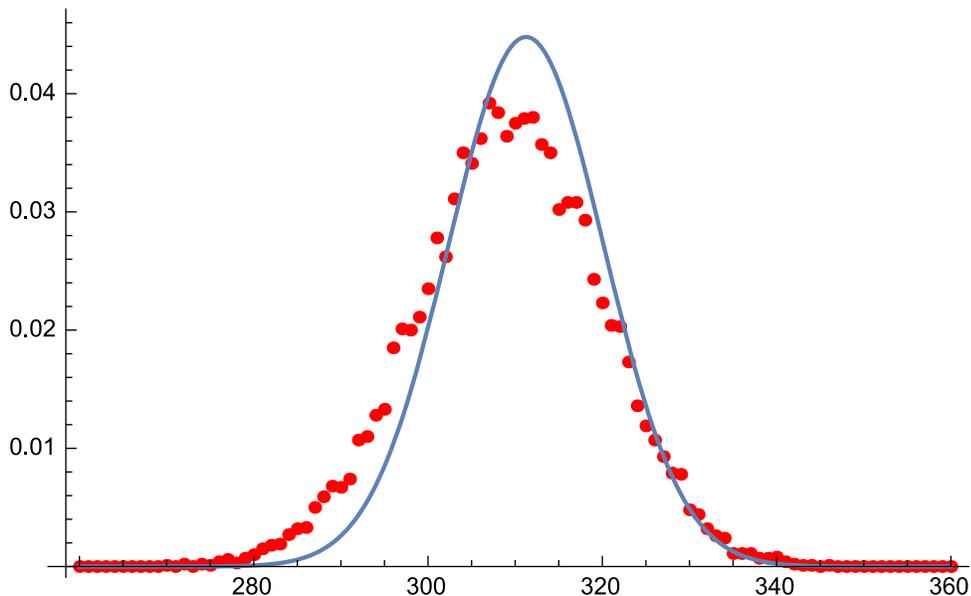}
	\caption
	    {\small The number of arcs in RNA secondary structures: we contrast the limit distribution of RNA
              secondary structures having minimum arc-length four and stack-length three (solid line) with
              the distribution of mfe-structures (dots),
              folded by \textsf{ViennaRNA}~\citep{Lorenz:11} from (uniformly sampled) random sequences of length $1000$.
              Here $n=1000$ to minimize finite size effects that would otherwise affect the distribution. 
	}
	\label{F:result}
\end{figure}

An RNA sequence is described by its primary structure, a linear oriented sequence of the nucleotides and can be viewed
as a string over the alphabet $\{\mathbf{A},\mathbf{U},\mathbf{G},\mathbf{C}\}$. An RNA strand folds by forming hydrogen
bonds between pairs of nucleotides according to Watson-Crick \textbf{A-U}, \textbf{C-G} and wobble \textbf{U-G} base-pairing
rules. The secondary structure encodes this bonding information of the nucleotides irrespective of the actual spacial embedding.
More than three decades ago, Waterman and colleagues pioneered the combinatorics and prediction of RNA secondary
structures~\citep{Waterman:78s,Waterman:79a,Waterman:78aa,Howell:80,Waterman:94a,Waterman:93}.

RNA secondary structures can be represented as diagrams, see Fig.~\ref{F:rnasec}. A \emph{diagram} is a labeled graph over
the vertex set $\{1, \dots, n\}$  whose vertices are arranged in a horizontal line and arcs are drawn in the upper half-plane.
Clearly, vertices and arcs correspond to nucleotides and base pairs, respectively. The number of nucleotides
is called the length of the structure. The length of an arc $(i,j)$ is defined as $j-i$ and an arc of length $k$ is called a $k$-arc.
The backbone of a diagram is the sequence of consecutive integers $(1,\dots,n)$ together with the edges $\{\{i,i+1\}\mid 1\le i\le n-1\}$.
We shall distinguish the backbone edge $\{i,i+1\}$ from the arc $(i,i+1)$, which we refer to as a \emph{$1$-arc}. 
Two arcs $(i_1,j_1)$ and $(i_2,j_2)$ are \emph{crossing} if $i_1<i_2<j_1<j_2$.  An
RNA \emph{secondary structure} is defined as a diagram satisfying the following three conditions~\citep{Waterman:78s}:
\begin{enumerate}
	\item \emph{non-existence of $1$-arcs}: if  $(i,j)$ is an arc, then $j-i\geq 2$,
	\item \emph{non-existence of base triples}: any two arcs do not have a common vertex, 
	\item \emph{non-existence of pseudoknots}: any two arcs are non-crossing, i.e., for two arcs $(i_1,j_1)$ and
	$(i_2,j_2)$ where $i_1<i_2$, $i_1<j_1$, and $i_2<j_2$, we have either $i_1<j_1<i_2<j_2$ or $i_1<i_2<j_2<j_1$.
\end{enumerate}
More succintly, an RNA secondary structure expressed as a diagram has exclusively noncrossig arcs in the upper half-plane.
Pairs of nucleotides may form Watson-Crick \textbf{A-U}, \textbf{C-G} and wobble \textbf{U-G} bonds labeling the above mentioned
arcs. 
\begin{figure}
	\centering
	\includegraphics[width=0.8\textwidth]{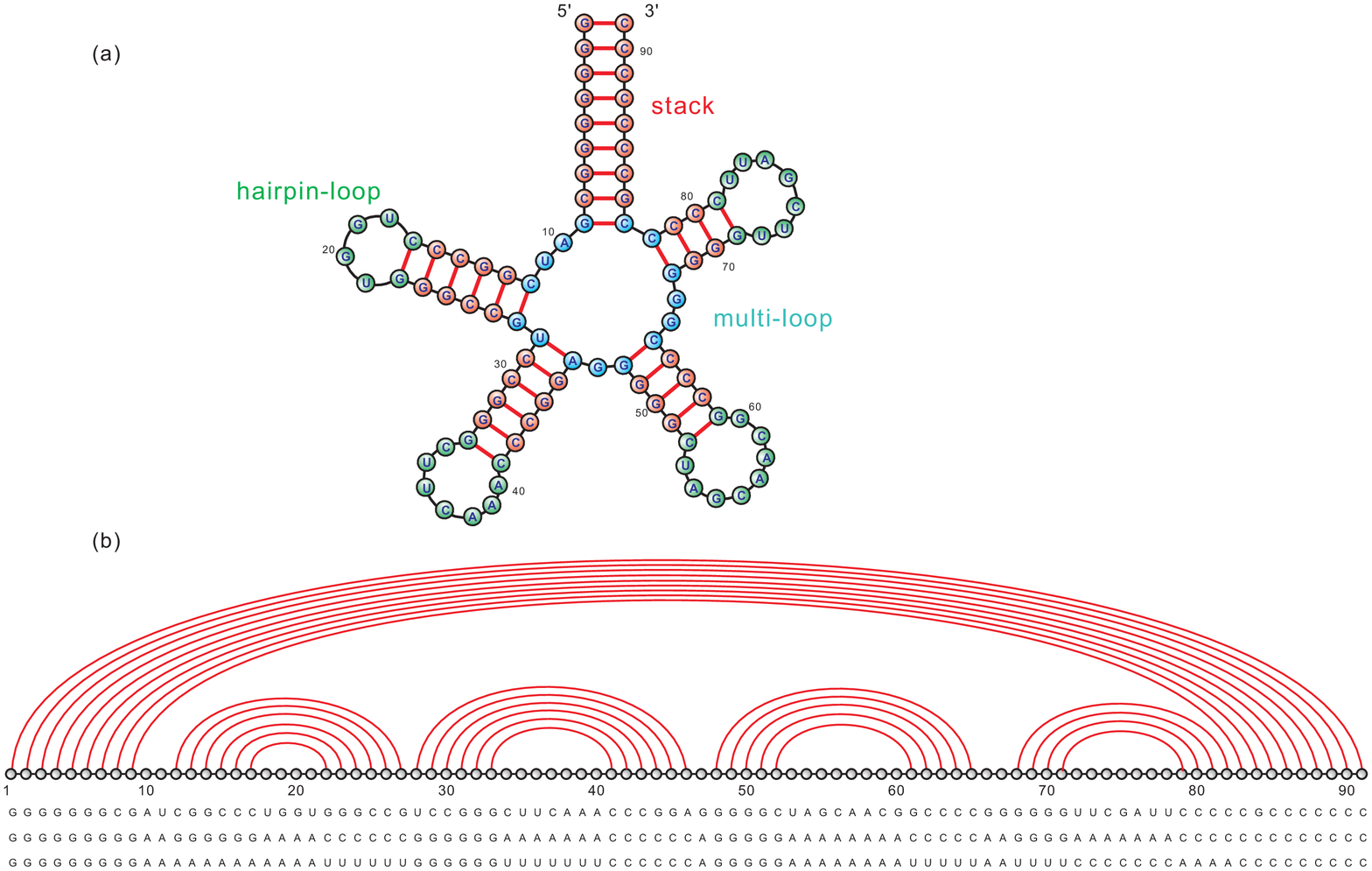}
	\caption
	{\small An RNA secondary structure represented as a contact graph (a) with hydrogen bonds as red edges
	  and as a diagram (b) having minimum stack-length $4$ and arc-length $5$. In (b), we display three compatible
          sequences
          having the $\mathbf{AUCG}$-ratios $(0.088,0.099,0.417,0.396)$, $(0.340,0.000,0.330,0.330)$ and 
$(0.296,0.242,0.242,0.220)$, respectively.
	}
	\label{F:rnasec}
\end{figure}

This paper is organized as follows: In Section~\ref{S:Back}, we provide some basic facts of singularity analysis of power series.
In Section~\ref{S:second}, we
compute the bivariate generating function of RNA secondary structures of length $n$ having $l$ arcs and derive its asymptotic
expansion
as well as the central limit theorem for the distribution of arcs. In Section~\ref{S:PT}, we analyze the fraction of
RNA secondary structures having compatible sequences and prove our main results. We conclude with Section~\ref{S:Dis},
where we integrate and discuss our findings.


\section{Singularity analysis}\label{S:Back}

RNA secondary structures as combinatorial objects give rise to generating functions,
i.e.~formal power series whose coefficients count the number of structures for given sequence length.
The estimation of these coefficients will be at the heart of this analysis and we next discuss
some basic facts of how to compute them. The key tool here is singularity
analysis: facilitated by Cauchy-integration along certain contours passing close by the dominant
critical point(s). Singularity analysis is described in great detail in the book  by Flajolet and
Sedgewick~\citep{Flajolet:07a}. In the following, we showcase all facts needed to exact the
asymptotics of coefficients from our generating functions.

Let $f(x)=\sum_{n\geq 0}a_n\, x^n$ be a combinatorial generating function. We are  interested in the
calculation of $a_n$ in terms of their exponential growth rate, $\gamma$ and their subexponential
factor $P(n)$, that is, $a_n\sim P(n)\,\gamma^n$.
The key for obtaining the asymptotic information about the coefficients of a generating function is
to locate its dominant singularities. For the particular case of power series with nonnegative coefficients
with a radius of convergence $R>0$, Pringsheim's theorem  \citep{Flajolet:07a,Tichmarsh:39} guarantees a
(unique) positive, real, dominant singularity at $z=R$. In our case all generating functions are by
construction combinatorial, whence we always have such a positive real, dominant, singularity.

A function $f(x)$ is $\Delta_\rho$ analytic at its dominant singularity $x=\rho$, if it is analytic in some
domain $\Delta_\rho(\phi,d)=\{ x\mid \vert x\vert < d, x\neq \rho,\, \vert
{\rm Arg}(x-\rho)\vert >\phi\}$, for some $\phi,d$, where $d>|\rho|$
and $0<\phi<\frac{\pi}{2}$. We set
\begin{equation*}
\left(f(x)=O\left(g(x)\right) \
\text{\rm as $x\rightarrow \rho$}\right)\  \Longleftrightarrow \
\left(f(x)/g(x)\ \text{\rm is bounded as $x\rightarrow \rho$}\right),
\end{equation*}
Noting that for any $\gamma\in\mathbb{C}\setminus 0$,
\begin{equation*}
[x^n]f(x)=\gamma^n
[x^n]f\left(\frac{x}{\gamma}\right),
\end{equation*}
we can, without loss of generality, reduce our analysis to the case
where $x=1$ is the unique dominant singularity. The following transfer-theorem allows us to derive the
asymptotics of coefficients from the asymptotic expansion of its generating function around its dominant
singularity.
\begin{theorem}[\citet{Flajolet:07a}, Theorem VI.3, pp. 390]\label{T:transfer1}
	Let $f(x)$ be a $\Delta_1$ analytic function at its unique dominant
	singularity $x=1$. Let
	\[g(x)=(1-x)^{\alpha}\log^{\beta}\left(\frac{1}{1-x}\right)
	,\quad\alpha,\beta\in \mathbb{R}.
	\]
	That is we have in the intersection of a neighborhood of $1$
	\begin{equation*}
	f(x) = O(g(x)) \quad \text{\it for } x\rightarrow 1.
	\end{equation*}
	Then we have
	\begin{equation*}
	[x^n]f(x)= O\left([x^n]g(x)\right).
	\end{equation*}
\end{theorem}
\begin{theorem}[\citet{Flajolet:07a}]\label{T:transfer2}
	Suppose $f(x)=(1-x)^{-\alpha}$, $\alpha\in\mathbb{C}\setminus
	\mathbb{Z}_{\le 0}$, then
	\begin{align*}
	[x^n]f(x) \sim & \frac{n^{\alpha-1}}{\Gamma(\alpha)}\left[
	1+\frac{\alpha(\alpha-1)}{2n}+\frac{\alpha(\alpha-1)(\alpha-2)(3\alpha-1)}
	{24 n^2}+\right. \\
	&
	\left.\frac{\alpha^2(\alpha-1)^2(\alpha-2)(\alpha-3)}{48n^3}+
	O\left(\frac{1}{n^4}\right)\right].
	\end{align*}
\end{theorem}

The next result extracts asymptotics of generating functions, satisfying polynomial
equations.
It is based on the method of Newton polygons and derives the Newton-Puiseux expansion and
the resultant of two polynomials to locate the dominant singularity. The resultant of two
polynomials $f(z)=a_n\prod_{i=1}^n (z-\alpha_i)$ and $g(z)=b_m\prod_{j=1}^m (z-\beta_j)$ is
given by
\[
\mathbf{R}(f,g,z)= a_n^m b_m^n\prod_{i=1}^n \prod_{j=1}^m (\alpha_i-\beta_j).
\]
\begin{theorem}\label{T:AsymG}
  Let $z(x)=\sum_{n\geq 0}z_n x^n$ be a generating function, analytic at $0$, that satisfies
  the polynomial equation $\Phi(x,z)=0$.  Let
	$
	\Delta(x)= \mathbf{R} \left(\Phi(x,z), \frac{\partial}{\partial z}\Phi(x,z),z
	\right)$. Suppose \\
	      {\bf(1)} there exist positive numbers $\rho$ and $\pi=z(\rho)$, such that  $\rho$ is a
              root of the resultant $\Delta(x)$ and $(\rho,\pi)$ satisfies the system of equations,
	\begin{equation}\label{E:phi1}
	\Phi(\rho,\pi)=0,\quad \Phi_z(\rho,\pi)=0,
	\end{equation}
	{\bf(2)} $\Phi(x,z)$ satisfies the conditions:
	\begin{equation}\label{E:phi2}
	\Phi_x(\rho,\pi)\neq 0,\quad \Phi_{z z}(\rho,\pi)\neq 0,
	\end{equation}	
        {\bf (3)} $z(x)$ is aperiodic, i.e., there exist three indices $i<j<k$ such that
        $z_i z_j z_k \neq 0$ and $\gcd(j-i,k-i)=1$.

        Then $\rho$ is the dominant singularity of  $z(x)$ and
	$z(x)$ has the following expansion at $\rho$
	\begin{equation}\label{E:asym1}
	z(x)=\pi+ \delta (\rho-x)^{\frac{1}{2}}+O(\rho-x),\quad \text{for some nonxero
		constant } \delta.
	\end{equation}
	Furthermore the coefficients of $z(x)$ satisfy
	\[
	[x^n]z(x) \sim  c \, n^{-\frac{3}{2}} \rho^{-n}, \quad n\rightarrow \infty,
	\]
	for some constant $c>0$.
\end{theorem}
The proof can be found in \citet{Flajolet:07a} or alternatively in \citet[pp. 103]{Hille:62}.
The key point here is, that when applying Newton's polygon method to determine the type of expansion
of $z(x)$: conditions (1) and (2) guarantee the first exponent of $x$ to be $\frac{1}{2}$. The asymptotics
of the coefficients follows then from eq.~(\ref{E:asym1}) as an  application of the transfer
theorem~\ref{T:transfer1} and theorem~\ref{T:transfer2}.


For a bivariate generating function $F(x,y)$, satisfying a functional equation, the next theorem allows us to
obtain key information about the singular expansion of $F(x,y)$, considering the latter as a univariate
generating function parameterized by $y$.
\begin{theorem}[\citet{Flajolet:07a}]\label{T:AsymAlgPra}
	Let $F(x,y)$ be a bivariate function that is analytic at $(0,0)$ and has
	non-negative coefficients. Suppose that $F(x,y)$ is one of the solutions $z$
	of a polynomial equation
	\[
	\Phi(x,y,z)=0,
	\]
	where $\Phi$ is a polynomial of degree $\geq 2$ in $z$, such that $\Phi(x,1,z)$ satisfies the
	conditions of Theorem~\ref{T:AsymG}. Let
	\[
	\Delta(x,y)= \mathbf{R} \left(\Phi(x,y,z),
	\frac{\partial}{\partial z}\Phi(x,y,z),z \right)
	\]
	and let $\rho$ be the root of $\Delta(x,1)$, such that $z(x):=F(x,1)$ is
	singular at $x =\rho $ and $z(\rho)=\pi$. Suppose
	there exists some $\rho(y)$ being the unique root of the equation
	\[
	\Delta(\rho(y),y)=0,
	\]
	analytic at $1$ and such that $\rho(1)=\rho$.

        Then $F(x,y)$ has the singular expansion
	\[
	F(x,y)= \pi(y) + \delta(y)\, \left(\rho(y)-x \right)^{\frac{1}{2}} \left(1+ o(1) \right),
	\]
	where $\pi(y)$ and $\delta(y)$ are analytic at 1 such that
	$\pi(1)=\pi$ and $\delta(1) \neq 0$. In addition we have
	\[
	[x^n]F(x,y) = c(y)\, n^{-\frac{3}{2}} \rho(y)^{-n}\left(1+O\left(\frac{1}{n}\right)\right),
	\]
	uniformly for $y$ restricted to a small neighborhood of $1$, where $c(y)$ is
	continuous and nonzero near 1.
\end{theorem}
Theorem~\ref{T:AsymAlgPra} is implied by Proposition IX. 17 and Theorem
IX. 12 of~\citet{Flajolet:07a}.



We shall end this section by stating the following central limit theorem due to Bender \citep{Bender:73}:
\begin{theorem}[\citet{Bender:73}]\label{T:normal}
	Suppose we are given the bivariate generating function
	\begin{equation*}
	f(x,u)=\sum_{n,t\geq 0}f(n,t)\,x^n\,u^t,
	\end{equation*}
	where $f(n,t)\geq 0$
	and $f(n)=\sum_tf(n,t)$. Let $\mathbb{X}_n$ be a r.v.~such that
	$\mathbb{P}(\mathbb{X}_n=t)=f(n,t)/f(n)$. Suppose
	\begin{equation*}
	[x^n]f(x,e^s)= c(s) \, n^{\alpha}\,  \gamma(s)^{-n}
	\left(1+ O\left( \frac{1}{n}\right) \right),
	\end{equation*}
	uniformly in $s$ in a neighborhood of $0$, where $c(s)$ is continuous
	and nonzero near $0$, $\alpha$ is a constant, and $\gamma(s)$ is analytic
	near $0$.
        
	Then there exists a pair $(\mu,\sigma)$ such that the normalized
	random variable
	\begin{equation*}
	\mathbb{X}^*_{n}=\frac{\mathbb{X}_{n}- \mu \, n}{\sqrt{{n\,\sigma}^2 }},
	\end{equation*}
	converges in distribution to a Gaussian variable with a speed of convergence
	$O(n^{-\frac{1}{2}})$.
	That is we have
	\begin{equation*}
	\lim_{n\to\infty}
	\mathbb{P}\left(
	\mathbb{X}^*_{n} < x \right)  =
	\frac{1}{\sqrt{2\pi}} \int_{-\infty}^{x}\,e^{-\frac{1}{2}t^2} \mathrm{d}t, \,
	\end{equation*}
	where $\mu$ and $\sigma^2$ are given by
	\begin{equation*}
	\mu= -\frac{\gamma'(0)}{\gamma(0)}
	\quad \text{\it and} \quad \sigma^2=
	\left(\frac{\gamma'(0)}{\gamma(0)}
	\right)^2-\frac{\gamma''(0)}{\gamma(0)}.
	\end{equation*}
\end{theorem}


\section{RNA secondary structures}\label{S:second}

We consider secondary structures subject to \emph{minimum arc-length} restrictions,
arising from the rigidity of the backbone. RNA secondary structures having minimum arc-length two
were studied by Waterman \citep{Waterman:78s}. Arguably, the most realistic cases is $\lambda=4$~\citep{Stein:79}, and
RNA folding algorithms, generating minimum free energy structures, implicitly satisfy this
constraint\footnote{each hairpin loop contains at least three unpaired bases} for energetic reasons.

A \emph{stack} of length $r$ is a maximal sequence of "parallel" arcs, $((i,j),(i+1,j-1),\ldots,(i+(r-1),j-(r-1)))$.
Stacks of length $1$
are energetically unstable and we find typically stacks of length at least two or three in biological structures~\citep{Waterman:78s}.
A secondary structure, $S$, is \emph{$r$-canonical} if its stack-length satisfies $\geq r$. 
An RNA sequence is \emph{compatible with a structure $S$}, if and only if for any $S$-arc $(i,j)$, the pair $(x_i,x_j)$ forms a
base pair \citep{Reidys:97}.

A \emph{rainbow} is an arc connecting the first and last vertices in a structure. A secondary structure is called
\emph{reducible} if it does not contain a rainbow. Let $ s_{\lambda}^{[r]}(n)$ and $ t_{\lambda}^{[r]}(n)$ denote the
numbers of $r$-canonical secondary structures
and reducible secondary structures  of $n$ nucleotides with minimum arc-length $\lambda$, respectively. Let furthermore
$s_{\lambda}^{[r]}(n,l)$  and $ t_{\lambda}^{[r]}(n,l)$ denote the numbers of $r$-canonical secondary structures and reducible
secondary structures, filtered by the number of arcs. Let $\mathbf{S}_{\lambda}^{[r]}(x,y)=\sum_{n,l} s_{\lambda}^{[r]}(n,l)x^n y^l$ and
$\mathbf{T}_{\lambda}^{[r]}(x,y)=\sum_{n,l} t_{\lambda}^{[r]}(n,l)x^n y^l$ denote the corresponding bivariate generating functions.

First we compute the generating function $\mathbf{S}_{\lambda}^{[r]}(x,y)$.
\begin{theorem}\label{T:arcgf}
	For any $\lambda, r\in \mathbb{N}$, the generating function $\mathbf{S}_{\lambda}^{[r]}(x,y)$ satisfies the functional equation
	\begin{equation}\label{Eq:arcfe}
	(x^2 y)^r  \mathbf{S}_{\lambda}^{[r]}(x,y)^2 -\mathbf{B}_{\lambda}^{[r]}(x,y)\, \mathbf{S}_{\lambda}^{[r]}(x,y)+\mathbf{A}^{[r]}(x,y)=0,
	\end{equation}
	where 
	\begin{align*}
	\mathbf{A}^{[r]}(x,y)&=1-x^2 y+(x^2 y)^r,\\
	\mathbf{B}_{\lambda}^{[r]}(x,y)&=(1-x) \mathbf{A}^{[r]}(x,y) + 	(x^2 y)^r\sum_{i=0}^{\lambda-2} x^i.
	\end{align*}
Explicitly, we have
\begin{equation}\label{Eq:arcunex}
	\begin{aligned}
	\mathbf{S}_{\lambda}^{[r]}(x,y)&= \frac{\mathbf{B}_{\lambda}^{[r]}(x,y)-\sqrt{\mathbf{B}_{\lambda}^{[r]}(x,y)^2-4 (x^2 y)^r \mathbf{A}^{[r]}(x,y)}}{2 (x^2 y)^r},\\
\mathbf{S}_{\lambda}^{[r]}(x,y)&= \frac{\mathbf{A}^{[r]}(x,y)}{\mathbf{B}_{\lambda}^{[r]}(x,y)}\, {\bf
	C}\!\left(\frac{(x^2 y)^r \mathbf{A}^{[r]}(x,y)}{\mathbf{B}_{\lambda}^{[r]}(x,y)^2}
\right),
	\end{aligned}
\end{equation}
	where	${\bf C}(x)$ denotes the generating function for the Catalan numbers and is given by $ \frac{1-\sqrt{1-4x}}{2x}$.
In particular, for $r=1$,	
	\begin{equation}\label{Eq:Nara}
	s_{\lambda}^{[1]}(n,l) =\sum_{k=1}^l N(l,k) \binom{n-(\lambda-1)k}{2l},
	\end{equation}
	where $N(l,k)= \frac{1}{l} \binom{l}{k} \binom{l}{k-1}$ is the Narayana number, counting the number of
	plane trees of $l$ edges with $k$ leaves.
\end{theorem}

\begin{proof} 
	First we derive a functional equation satisfied by  $\mathbf{S}_{\lambda}^{[r]}(x,y)$ and $\mathbf{T}_{\lambda}^{[r]}(x,y)$ via a decomposition for an $r$-canonical secondary structure with minimum arc-length $\lambda$. Any given structure must belong to one of the following three classes: 
	\begin{enumerate}
		\item the empty structure it corresponds to the coefficient $1$;
		\item the collection of structures, starting with an unpaired vertex, are counted by
		$x \mathbf{S}_{\lambda}^{[r]}(x,y)$;
		\item structures starting with a paired vertex induce a maximum stack containing this base pair, and two segments, one being the nested
		substructure and the other the concatenated substructure.
	        An arc corresponds to $x^2 y$ and a stack of size at least $r$ corresponds to $\frac{(x^2 y)^r}{1-x^2 y}$.
                The nested segment must by construction be reducible and contains at
		least $\lambda-1$ vertices, as a result of the arc-length restriction.
		Since any segment with at most $\lambda-2$ vertices corresponds to the term
		$\sum_{i=0}^{\lambda -2}x^i$, the nested segment gives rise to
		$\mathbf{T}_{\lambda}^{[r]}(x,y)-\sum_{i=0}^{\lambda -2}x^i$.
		The concatenated segment gives rise to the term $\mathbf{S}_{\lambda}^{[r]}(x,y)$. Thus
		we arrive at
		\[\frac{(x^2 y)^r}{1-x^2 y} \, \mathbf{S}_{\lambda}^{[r]}(x,y) \Big(\mathbf{T}_{\lambda}^{[r]}(x,y)-\sum_{i=0}^{\lambda -2}x^i\Big).\]
	\end{enumerate}
	The decomposition is illustrated in Fig.~\ref{F:dec}.
	Accordingly we obtain the functional equation:
	\begin{equation}\label{Eq:1}
	\mathbf{S}_{\lambda}^{[r]}(x,y)=1+x \mathbf{S}_{\lambda}^{[r]}(x,y)+\frac{(x^2 y)^r}{1-x^2 y} \, \mathbf{S}_{\lambda}^{[r]}(x,y) \Big(\mathbf{T}_{\lambda}^{[r]}(x,y)-\sum_{i=0}^{\lambda -2}x^i\Big).
	\end{equation}
        Now we decompose an irreducible secondary structure deriving a second functional equation, relating $\mathbf{S}_{\lambda}^{[r]}(x,y)$ and
        $\mathbf{T}_{\lambda}^{[r]}(x,y)$. Given an irreducible structure, it can decomposed as a maximum stack containing its rainbow and the
        nested substructure, see Fig.~\ref{F:dec}. The nested segment has to be reducible and has length at least $\lambda-1$,
        i.e.~$\mathbf{T}_{\lambda}^{[r]}(x,y)-\sum_{i=0}^{\lambda -2}x^i$. The irreducible structures have the generating function
        $\mathbf{S}_{\lambda}^{[r]}(x,y)-\mathbf{T}_{\lambda}^{[r]}(x,y)$, whence
\begin{equation}\label{Eq:2}
\mathbf{S}_{\lambda}^{[r]}(x,y)-\mathbf{T}_{\lambda}^{[r]}(x,y)= \frac{(x^2 y)^r}{1-x^2 y}\Big(\mathbf{T}_{\lambda}^{[r]}(x,y)-\sum_{i=0}^{\lambda -2}x^i\Big).
\end{equation}
Solving eq.~(\ref{Eq:2}) for $\mathbf{T}_{\lambda}^{[r]}(x,y)$ and substituting the solution into eq.~(\ref{Eq:1}), we obtain eq.~(\ref{Eq:arcfe})
	which immediately implies eqs.~(\ref{Eq:arcunex}).
	\begin{figure}
		\centering
		\includegraphics[width=0.8\textwidth]{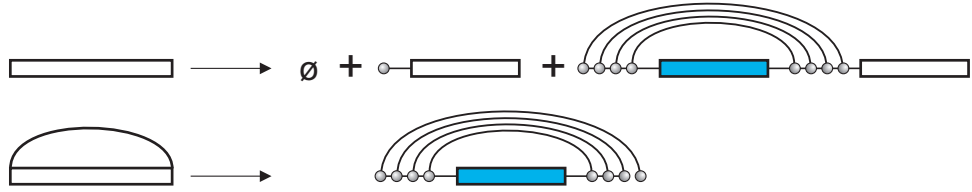}
		\caption
		{\small The decomposition of secondary structure and irreducible structure (reducible structures are colored in blue).}
		\label{F:dec}
	\end{figure}
	It remains to prove eq.~(\ref{Eq:Nara}).
	Given a secondary structure $S$ having $n$ vertices and $l$ arcs, removing any unpaired vertices induces the
	diagram $S'$, i.e.~a matching. Suppose $S'$ has $k$ $1$-arcs. It is well-known that the set of matchings with
	$l$ arcs and $k$ $1$-arcs corresponds to the set of plane trees of $l$ edges with $k$ leaves, which
	is counted by the Narayana number $N(l,k)= \frac{1}{l} \binom{l}{k} \binom{l}{k-1}$ (see e.g. \citet{stanley:2001}).
        In view of minimum arc-length constraints, in order to recover $S$ from $S'$, we need to guarantee that each $1$-arc
	contains at least $\lambda-1$ unpaired vertices. Accordingly, the number of all possible insertions is given
	by $ \binom{n-(\lambda-1)k}{2l}$, i.e~eq.~(\ref{Eq:Nara}).
\end{proof}

{\bf Remark:} Eq.~(\ref{Eq:Nara}) generalizes the closed formula $
s_{2}^{[1]}(n,l) =\frac{1}{l} \binom{n-l}{l+1} \binom{n-l-1}{l-1}
$ proved in~\citet{Waterman:94a} by establishing a bijection between secondary structures and linear trees.





Interpreting the indeterminant $y$ as a parameter, we consider $\mathbf{S}_{\lambda}^{[r]}(x,y)$ as a univariate
powerseries. One key question then is to compute the asymptotic behaviour of the coefficients.

\begin{theorem}\label{T:asyexp}
	For $1\leq \lambda \leq 4$ and $1\leq r \leq 3$, the coefficients of $\mathbf{S}_{\lambda}^{[r]}(x,y)$ are asymptotically given by
	\begin{equation} \label{Eq:arcsing}
	[x^{n} ] \mathbf{S}_{\lambda}^{[r]}(x,y) =c_{\lambda}^{[r]}(y) n^{-\frac{3}{2} } \big(\rho_{\lambda}^{[r]}(y)
	\big)^{-n}\ \big(1+O(n^{-1})\big),
	\end{equation}
	as $n \rightarrow \infty$, uniformly, for $y$ restricted to a small neighborhood of $1$, where
	$c_{\lambda}^{[r]}(y)$ is continuous and nonzero near $1$, and $\rho_{\lambda}^{[r]}(y)$ is the minimal positive,
	real solution of $\mathbf{B}_{\lambda}^{[r]}(x,y)^2 -4 (x^2 y)^r \mathbf{A}^{[r]}(x,y)=0$, for $y$ in a neighborhood of $1$.
	
	In particular, we have
	\[
	\rho_1^{[1]}(y)=\frac{1-2\sqrt{y}}{1-4y},\quad \rho_2^{[1]}(y)=\frac{1+2 \sqrt{y}-\sqrt{1+4 \sqrt{y}}}{2 y}.
	\]
\end{theorem}

\begin{proof}
	We shall employ Theorem~\ref{T:AsymAlgPra}.
	
	{\bf Step $1$:} we consider first the monovariate power series
	$\mathbf{S}_{\lambda}^{[r]}(x,1)$: 
	
	In view of eq.~(\ref{Eq:arcfe}) of Theorem~\ref{T:arcgf} and setting $y=1$, we see that
	$\mathbf{S}_{\lambda}^{[r]}(x,1)$ satisfies the equation 
	\[
	\Phi(x,\mathbf{S}_{\lambda}^{[r]}(x,1))=
		x^{2r}  \mathbf{S}_{\lambda}^{[r]}(x,y)^2 -\mathbf{B}_{\lambda}^{[r]}(x,1)\, \mathbf{S}_{\lambda}^{[r]}(x,1)+\mathbf{A}^{[r]}(x,1)
=0,
	\]
	where $\Phi(x,z)=x^{2r} z^2- \mathbf{B}_{\lambda}^{[r]}(x,1) z+\mathbf{A}^{[r]}(x,1)$. Next we observe that  
	\[
	\Delta(x)= \mathbf{R} \left(\Phi(x,z), \frac{\partial}{\partial z}\Phi(x,z),z
	\right)= -x^{2r} ( \mathbf{B}_{\lambda}^{[r]}(x,1)^2 -4 x^{2r}\mathbf{A}^{[r]}(x,1) ),
	\]
	where $\mathbf{R}(f,g,z)$ is the resultant of two polynomials $f(z)=a_n\prod_{i=1}^n (z-\alpha_i)$
	and $g(z)=b_m\prod_{j=1}^m (z-\beta_j)$, $\mathbf{R}(f,g,z)= a_n^m b_m^n\prod_{i=1}^n \prod_{j=1}^m
	(\alpha_i-\beta_j)$.
	
	Ad {\bf (1)}, we inspect for $1\leq \lambda \leq 4$ and $1\leq r \leq 3$, that there exist
	positive numbers $\rho_{\lambda}^{[r]}$ and $\pi_{\lambda}^{[r]}=\mathbf{S}_{\lambda}^{[r]}(\rho_{\lambda}^{[r]},1)$,
	such that $\rho_{\lambda}^{[r]}$ is a root of the resultant $\Delta(x)$ and $(\rho_{\lambda}^{[r]},\pi_{\lambda}^{[r]})$
	and furthermore $\Phi(\rho_{\lambda}^{[r]},\pi_{\lambda}^{[r]})=0$ as well as $\Phi_z(\rho_{\lambda}^{[r]},\pi_{\lambda}^{[r]})=0$.
	In fact, $\rho_{\lambda}^{[r]}$ is the minimal positive real solution of the $\Delta(x)$-divisor
	$ \mathbf{B}_{\lambda}^{[r]}(x,1)^2 -4 x^{2r}\mathbf{A}^{[r]}(x,1) $.
	
	As for {\bf (2)}, we verify $\Phi_x(\rho_{\lambda}^{[r]},\pi_{\lambda}^{[r]})\neq 0$ and
	$\Phi_{z z}(\rho_{\lambda}^{[r]},\pi_{\lambda}^{[r]})\neq 0$.
	
	Ad {\bf (3)}, for any $\lambda$ there exists some $i$, such that $s_{\lambda}^{[r]}(i),
	s_{\lambda}^{[r]}(i+1),s_{\lambda}^{[r]}(i+2)\neq 0$, i.e.~$\mathbf{S}_{\lambda}^{[r]}(x,1)$ is aperiodic.
	
	Theorem~\ref{T:AsymG} implies that $\rho_{\lambda}^{[r]}$ is the dominant singularity and that the singular
	expansion of $\mathbf{S}_{\lambda}^{[r]}(x,1)$ is given by 
	\[
	\mathbf{S}_{\lambda}^{[r]}(x,1)= \pi_{\lambda}^{[r]}+ \delta_{\lambda}^{[r]} \big( \rho_{\lambda}^{[r]}-x \big)^{\frac{1}{2}} (1+o(1)),
	\]
	where $\delta_{\lambda}^{[r]}$ is a positive real number. 
	
	{\bf Step $2$:} we are now in position to employ Theorem~\ref{T:AsymAlgPra}.
	By eq.~(\ref{Eq:arcfe}), $\mathbf{S}_{\lambda}^{[r]}(x,y)$ satisfies 
	\begin{equation*}
(x^2 y)^r  \mathbf{S}_{\lambda}^{[r]}(x,y)^2 -\mathbf{B}_{\lambda}^{[r]}(x,y)\, \mathbf{S}_{\lambda}^{[r]}(x,y)+\mathbf{A}^{[r]}(x,y)=0
	\end{equation*}
	and $\Phi(x,y,z)=(x^2 y)^r z^2 -\mathbf{B}_{\lambda}^{[r]}(x,y) z+\mathbf{A}^{[r]}(x,y)$ is a polynomial in $z$ of degree two.
	$
	\Delta(x,y)= \mathbf{R} \left(\Phi(x,y,z), \frac{\partial}{\partial z}\Phi(x,y,z),z\right)
	$
	satisfies $\Delta(x,1)=\Delta(x)$ and $\rho_{\lambda}^{[r]}$ is a root of $\Delta(x,1)$.
	Step $1$ implies that $z(x)=\mathbf{S}(x,1)$ has the unique dominant singularity $\rho_\lambda$
	and $z(\rho)=\pi_\lambda$.
	We compute
	\[
	\Delta(x,y)= \mathbf{R} \left(\Phi(x,y,z), \frac{\partial}{\partial z}\Phi(x,y,z),z
	\right)= -(x^2 y)^r  ( \mathbf{B}_{\lambda}^{[r]}(x,y)^2 -4 (x^2 y)^r \mathbf{A}^{[r]}(x,y))
	\]
	and check that there exists for $1\leq \lambda \leq 4$ and $1\leq r \leq 3$ a unique $\rho_{\lambda}^{[r]}(y)$,
	such that $\Delta(\rho_{\lambda}^{[r]}(y),y)=0$ and $\rho_{\lambda}^{[r]}(1)=\rho_{\lambda}^{[r]}$.
	$\rho_{\lambda}^{[r]}(y)$ is the minimal, positive, real solution of $\mathbf{B}_{\lambda}^{[r]}(x,y)^2 -4 (x^2 y)^r \mathbf{A}^{[r]}(x,y)=0$
	for $y$ in a neighborhood of $1$.
	
	By Theorem~\ref{T:AsymAlgPra}, the singular expansion of $\mathbf{S}_{\lambda}^{[r]}(x,y)$ is given by
	\[
	\mathbf{S}_{\lambda}^{[r]}(x,y)= \pi_{\lambda}^{[r]}(y)+ \delta_{\lambda}^{[r]}(y) \big( \rho_{\lambda}^{[r]}(y)-
	x \big)^{\frac{1}{2}} (1+o(1)),
	\]
	where $\pi_{\lambda}^{[r]}(y)$ and $\delta_{\lambda}^{[r]}(y)$ are analytic at $1$ and
	$\delta_{\lambda}^{[r]}(1)$ is a positive, real number and furthermore
	$$
	[x^{n} ] \mathbf{S}_{\lambda}^{[r]}(x,y) =c_{\lambda}^{[r]}(y) n^{-\frac{3}{2} } \big(\rho_{\lambda}^{[r]}(y)
	\big)^{-n}\ \big(1+O(n^{-1})\big).
	$$ 
\end{proof}



We next analyze the random variable $\mathbb{Y}_{\lambda,n}^{[r]}$, counting the numbers of arcs in RNA secondary
structures. By construction we have
\[
\mathbb{P}(\mathbb{Y}_{\lambda,n}^{[r]}=l)=\frac{s_{\lambda}^{[r]}(n,l)}{s_{\lambda}^{[r]}(n)},
\]
where $l=0,1,\ldots, \lfloor\frac{n+1-\lambda}{2} \rfloor $.

Theorem~\ref{T:asyexp} and Theorem~\ref{T:AsymAlgPra} immediately imply

\begin{theorem}\label{T:arcclt}
	There exists a pair $(\mu_{\lambda}^{[r]},\sigma_{\lambda}^{[r]})$ such that the normalized random variable
	\begin{equation*}
	\mathbb{Y}^{[r],*}_{n,\lambda}=\frac{\mathbb{Y}_{\lambda,n}^{[r]}- \mu_{\lambda}^{[r]} \, n}{\sqrt{n\,
			(\sigma_{\lambda}^{[r]})^2 }}
	\end{equation*}
	converges in distribution to a Gaussian variable with a speed of convergence $O(n^{-\frac{1}{2}})$. That
	is, we have
	\begin{equation}\label{Eq:sclt}
	\lim_{n\to\infty}\mathbb{P}\left(\frac{\mathbb{Y}_{\lambda,n}^{[r]}- \mu_{\lambda}^{[r]}
		n}{\sqrt{n\, (\sigma_{\lambda}^{[r]})^2}} < x \right)  =
	\frac{1}{\sqrt{2\pi}}\int_{-\infty}^{x}\,e^{-\frac{1}{2}t^2} \mathrm{d}t \ ,
	\end{equation}
	where $\mu_{\lambda}^{[r]}$ and $\sigma_{\lambda}^{[r]}$ are given by
	\begin{equation}\label{Eq:cltpara}
	\mu_{\lambda}^{[r]}= -\frac{\theta'(0)}{\theta(0)},
	\qquad \qquad (\sigma_{\lambda}^{[r]})^2=
	\left(\frac{\theta'(0)}{\theta(0)}
	\right)^2-\frac{\theta''(0)}{\theta(0)},
	\end{equation}
	where $\theta(s) = \rho_{\lambda}^{[r]}(e^s)$.
\end{theorem}

In Table~\ref{Tab:clt}, we list the values of $\mu_{\lambda}^{[r]}$ and $(\sigma_{\lambda}^{[r]})^2$ for $1\leq \lambda \leq 4$
and $1\leq r \leq 3$.

{\bf Remark.} The expectation and variance of the number of arcs in secondary structures has been studied
computationally in~\citet{Fontana:93} for mfe structures and combinatorially in~\citet{Hofacker:98}
for $1$-canonical structures with minimum arc-length $\lambda$. In \citet{Jin-Reidys} more general results
on local and global limit theorems for $k$-noncrossing structures are proved. These imply, setting $k=2$,
Theorem~\ref{T:arcclt}. The above approach, in light of the combinatorial proof of the recursion
in Theorem~\ref{T:arcgf}, provides an elementary derivation.

\begin{table}[htbp]
	\caption{The central limit theorem for the number of arcs.
		We list the values of $\mu_{\lambda}^{[r]}$ and $(\sigma_{\lambda}^{[r]})^2$ derived from eq.~(\ref{Eq:cltpara}).} \label{Tab:clt}
	\begin{center}
		\begin{tabular}{ccccccc}
			\hline
			&\multicolumn{2}{c}{$r=1$} &\multicolumn{2}{c}{$r=2$}
			&\multicolumn{2}{c}{$r=3$}  \\
			\hline
			& $\mu_{\lambda}^{[r]}$ & $(\sigma_{\lambda}^{[r]})^2$
			& $\mu_{\lambda}^{[r]}$ & $(\sigma_{\lambda}^{[r]})^2$
			& $\mu_{\lambda}^{[r]}$ & $(\sigma_{\lambda}^{[r]})^2$\\
			\hline
			$\lambda=1$ & $0.3333$ & $0.0556$ & $0.3484$ & $0.0719$ & $0.3582$ & $0.0852$\\
			$\lambda=2$ & $0.2764$ & $0.0447$ & $0.3172$ & $0.0643$ & $0.3364$ & $0.0791$ \\
	        $\lambda=3$ & $0.2500$ & $0.0442$ & $0.2983$ & $0.0631$ & $0.3215$ & $0.0778$ \\
	        $\lambda=4$ & $0.2367$ & $0.0469$ & $0.2865$ & $0.0651$ & $0.3113$ & $0.0793$ \\
			\hline
		\end{tabular}
	\end{center}
\end{table}

Theorem~\ref{T:arcclt} follows from Theorem~\ref{T:asyexp} and Theorem~\ref{T:AsymAlgPra}, setting
$f(x,e^s)=\mathbf{S}_{\lambda}^{[r]}(x,e^s)$. 

We display the distribution of the numbers of arcs for sequence length $n=400$ in Fig.~\ref{F:unclt},
observing that the expected number of arcs drops from $111$ to $95$, when increasing the minimum
arc-length from $2$ to $4$.

\begin{figure}
	\centering
	\includegraphics[width=0.6\textwidth]{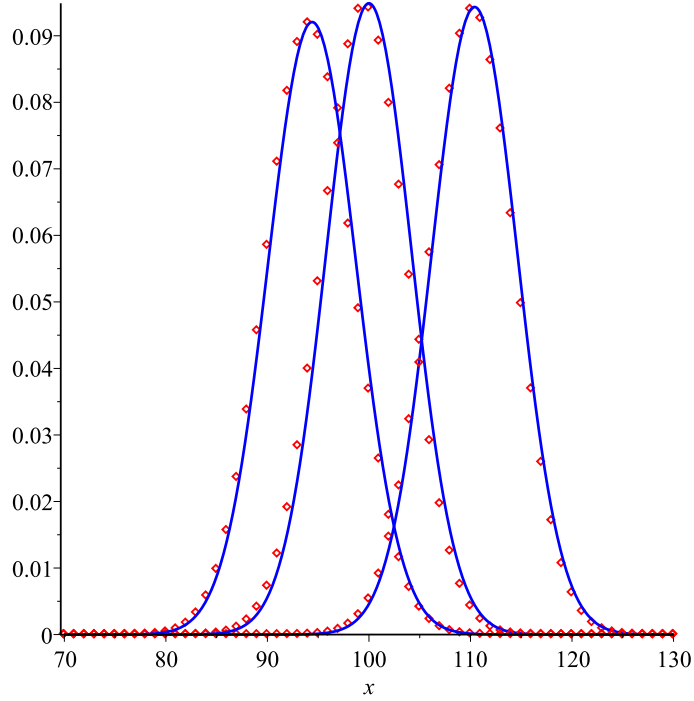}
	\caption
	    {\small The distribution of the number of arcs of $1$-canonical secondary structures of length $n=400$ for
              $\lambda=2$(R), $\lambda=3$(M), $\lambda=4$(L).
	}
	\label{F:unclt}
\end{figure}


\section{Main results}\label{S:PT}



\subsection{Structures having Purine-Pyrimidine base pairs}


We consider RNA structures, $S$, over sequences having only two types of nucleobases: purines ($R$) and
pyrimidines ($Y$), having the base pairs $(R,Y)$ and $(Y,R)$.

For any $p\in ]0,1[$ and $n\in \mathbb{N}$, let $s_{p,\lambda}^{[r]}(n)$ denote the number of $r$-canonical RNA
    secondary structures for which there exists at least one compatible sequence, containing
    $\lfloor  p n \rfloor$ purine bases and $\lceil(1-p)n\rceil$ pyrimidine bases, respectively. Let
    furthermore $s_{p,\lambda}^{[r]}(n,l)$ denote the corresponding number of structures, filtered by the number of arcs.

\begin{remark}
	In view of $\lfloor  p n \rfloor/{n} =p+O(n^{-1})$, $s_{p,\lambda}^{[r]}(n)$ or $s_{p,\lambda}^{[r]}(n,l)$, count in
	the limit of large sequence lengths, secondary structures of sequences whose percentage of purine
	bases equals $p$.
\end{remark}

\begin{proposition}
	For any $n,r,l,\lambda\in \mathbb{N}$ and $p\in ]0,1[$, we have 
	\begin{equation}\label{Eq:pp}
	\begin{aligned}
	s_{p,\lambda}^{[r]}(n,l)&= s_{\lambda}^{[r]}(n,l) \  \quad \text{if } l\leq
	\min( \lfloor p n\rfloor ,\lceil (1-p) n\rceil),\\
	s_{p,\lambda}^{[r]}(n,l)&= 0 \qquad \qquad \text{if } l>  \min( \lfloor p n\rfloor ,
	\lceil (1-p) n\rceil,
	\end{aligned}
	\end{equation}
	and
	\begin{equation}\label{Eq:ppsum}
	\begin{aligned}
	s_{p,\lambda}^{[r]}(n)&= \sum_{l=0}^{\min( \lfloor p n \rfloor,\lceil (1-p) n\rceil) } s_{\lambda}^{[r]}(n,l).
	\end{aligned}
	\end{equation}
\end{proposition}

\begin{proof}
	We shall show that any such structure having $l$ arcs is counted by $s_{p,\lambda}^{[r]}(n,l)$ iff
	$l\le \lfloor pn\rfloor$ and
	$l\le \lceil (1-p)n\rceil$. 
	Given a structure $S$ having $l$ arcs, each arc corresponds to a pair $(R,Y)$ or $(Y,R)$ for any
	compatible sequence,
	while each unpaired vertex could be either $R$ or $Y$.
	Accordingly, in any compatible sequence, the number of purine
	bases is at least $l$.
	For $S$ to be counted by $s_{p,\lambda}^{[r]}(n,l)$ the existence of $l$ arcs thus implies $l\le \lfloor pn\rfloor$. Since we have
	$\lceil (1-p)n\rceil$ pyrimidine bases in such a sequence, we observe for the same reason $l\le \lceil (1-p)n\rceil$.
	In case of $l>  \min( \lfloor p n\rfloor , \lceil (1-p) n\rceil)$ we either have not enough purine or pyrimidine bases,
	whence $s_{p,\lambda}^{[r]}(n,l)= 0$ and the proposition follows.
\end{proof}


\begin{theorem}\label{T:ppclt}
	For $1\leq \lambda \leq 4$, $1\leq r\leq 3$, $p\in ]0,\frac{1}{2}]$ and $n\in\mathbb{N}$, we have
	\begin{equation}\label{Eq:estimate}
	\frac{s_{p,\lambda}^{[r]}(n)}{s_{\lambda}^{[r]}(n)}=
	\frac{1}{\sqrt{2\pi}}\int_{-\infty}^{\frac{(p-\mu_{\lambda}^{[r]})\sqrt{n}}{\sigma_{\lambda}^{[r]}}}\,e^{-\frac{1}{2}t^2} \mathrm{d} t +
	O(n^{-\frac{1}{2}}),
	\end{equation}
	where $\mu_\lambda^{[r]},\sigma_\lambda^{[r]}$ are the mean and standard deviation of $\mathbb{Y}_{\lambda,n}^{[r]}$,
        respectively.

        Equivalently, for $p\in ]\mu_\lambda^{[r]},\frac{1}{2}]$,
        a random structure a.a.s. has a compatible sequence with nucleotide ratio $p$. Conversely, in case of
        $p\in]0,\mu_\lambda^{[r]}[$, a.a.s. no random structure has a compatible sequence.
\end{theorem}
\begin{proof}
	For $p\in ]0,\frac{1}{2}]$, we have
	\begin{align*}
	  \frac{s_{p,\lambda}^{[r]}(n)}{s_{\lambda}^{[r]}(n)}&=
          \sum_{l=0}^{ \min( \lfloor p n\rfloor ,\lceil (1-p) n\rceil) } \frac{s_{\lambda}^{[r]}(n,l)}{s_{\lambda}^{[r]}(n)} 
	= \sum_{l=0}^{ \lfloor p n\rfloor } \frac{s_{\lambda}^{[r]}(n,l)}{s_{\lambda}^{[r]}(n)}
	= \mathbb{P}(\mathbb{Y}_{\lambda,n}^{[r]}\leq  \lfloor p n\rfloor ),
	\end{align*}
	where the first equation employs eq.~(\ref{Eq:pp}).
	Theorem~\ref{T:arcclt} allows us to estimate $\mathbb{P}(\mathbb{Y}_{\lambda,n}^{[r]}\leq  \lfloor p n\rfloor )$ by the integral
        of the Gaussian density function with an $O(n^{-\frac{1}{2}})$ error term.
	Specifically, setting $x= \frac{\lfloor pn \rfloor -\mu_{\lambda}^{[r]}n }{\sigma_{\lambda}^{[r]}\sqrt{n}}$
	in eq.~(\ref{Eq:sclt}), we obtain 
	$$
	  \mathbb{P}(\mathbb{Y}_{\lambda,n}^{[r]}\leq  \lfloor p n\rfloor )=
          \frac{1}{\sqrt{2\pi}}\int_{-\infty}^{\frac{\lfloor pn \rfloor -\mu_{\lambda}^{[r]}n }{\sigma_{\lambda}^{[r]}\sqrt{n}}}\,e^{-\frac{1}{2}t^2} \mathrm{d} t
          + O(n^{-\frac{1}{2}})
	=\frac{1}{\sqrt{2\pi}}\int_{-\infty}^{\frac{(p-\mu_{\lambda}^{[r]})\sqrt{n}}{\sigma_{\lambda}^{[r]}}}\,e^{-\frac{1}{2}t^2} \mathrm{d} t + O(n^{-\frac{1}{2}}), 
	$$
	where the last equality is implied by
	\begin{align*}
	\Big|\int_{-\infty}^{\frac{(p-\mu_{\lambda}^{[r]})\sqrt{n}}{\sigma_{\lambda}^{[r]}}}\,e^{-\frac{1}{2}t^2} \mathrm{d} t -\int_{-\infty}^{\frac{\lfloor pn \rfloor -\mu_{\lambda}^{[r]}n }{\sigma_{\lambda}^{[r]}\sqrt{n}}}\,e^{-\frac{1}{2}t^2} \mathrm{d} t \Big|& =\Big| \int_{\frac{\lfloor pn \rfloor -\mu_{\lambda}^{[r]}n }{\sigma_{\lambda}^{[r]}\sqrt{n}}}^{\frac{pn-\mu_{\lambda}^{[r]}n}{\sigma_{\lambda}^{[r]}\sqrt{n}}}\,e^{-\frac{1}{2}t^2} \mathrm{d} t\Big|\\
	&\le \frac{1}{\sigma_{\lambda}^{[r]}\sqrt{n}}\cdot e^{-\frac{1}{2}t^2}\Big|_{t=0} = O(n^{-\frac{1}{2}}),
	\end{align*}
	whence eq.~(\ref{Eq:estimate}). {For $p\in]\mu_\lambda^{[r]},\frac{1}{2}]$} 
	\[
	\frac{s_{p,\lambda}^{[r]}(n)}{s_{\lambda}^{[r]}(n)}\geq
	\frac{1}{\sqrt{2\pi}}\int_{-\infty}^{\frac{\epsilon\sqrt{n}}{\sigma_{\lambda}^{[r]}}}\,e^{-\frac{1}{2}t^2} \mathrm{d} t +
	O(n^{-\frac{1}{2}})\rightarrow 1,  \qquad \text{as }  n\rightarrow \infty,	
	\] implying that asymptotically, almost surely any structure has a compatible sequence with nucleotide ratio $p$
        and the theorem follows. 
\end{proof}

For any Gaussian distributed random variable, $\mathbb{X}$, we have
\begin{align*}
  \mathbb{P}(\mu-2\sigma \leq \mathbb{X}\leq\mu+2\sigma ) &=\frac{1}{\sigma\sqrt{2\pi}}\int_{\mu-2\sigma}^{\mu+2\sigma}\,e^{-\frac{(t-\mu)^2}{2\sigma^2}}
  \approx 0.9545,\\
  \mathbb{P}(\mu-3\sigma \leq \mathbb{X}\leq\mu+3\sigma ) &=\frac{1}{\sigma\sqrt{2\pi}}\int_{\mu-3\sigma}^{\mu+3\sigma}\,e^{-\frac{(t-\mu)^2}{2\sigma^2}}
  \approx 0.9973,
\end{align*}
i.e.~the fraction of its values within two and three standard deviations away from its mean is approximately $0.9545$ and
$0.9973$, respectively. In Fig.~\ref{F:threshold} we display the limit distribution.
 \begin{figure*}
 	\begin{tabular}{cc}
 		\includegraphics[width=0.5\textwidth]{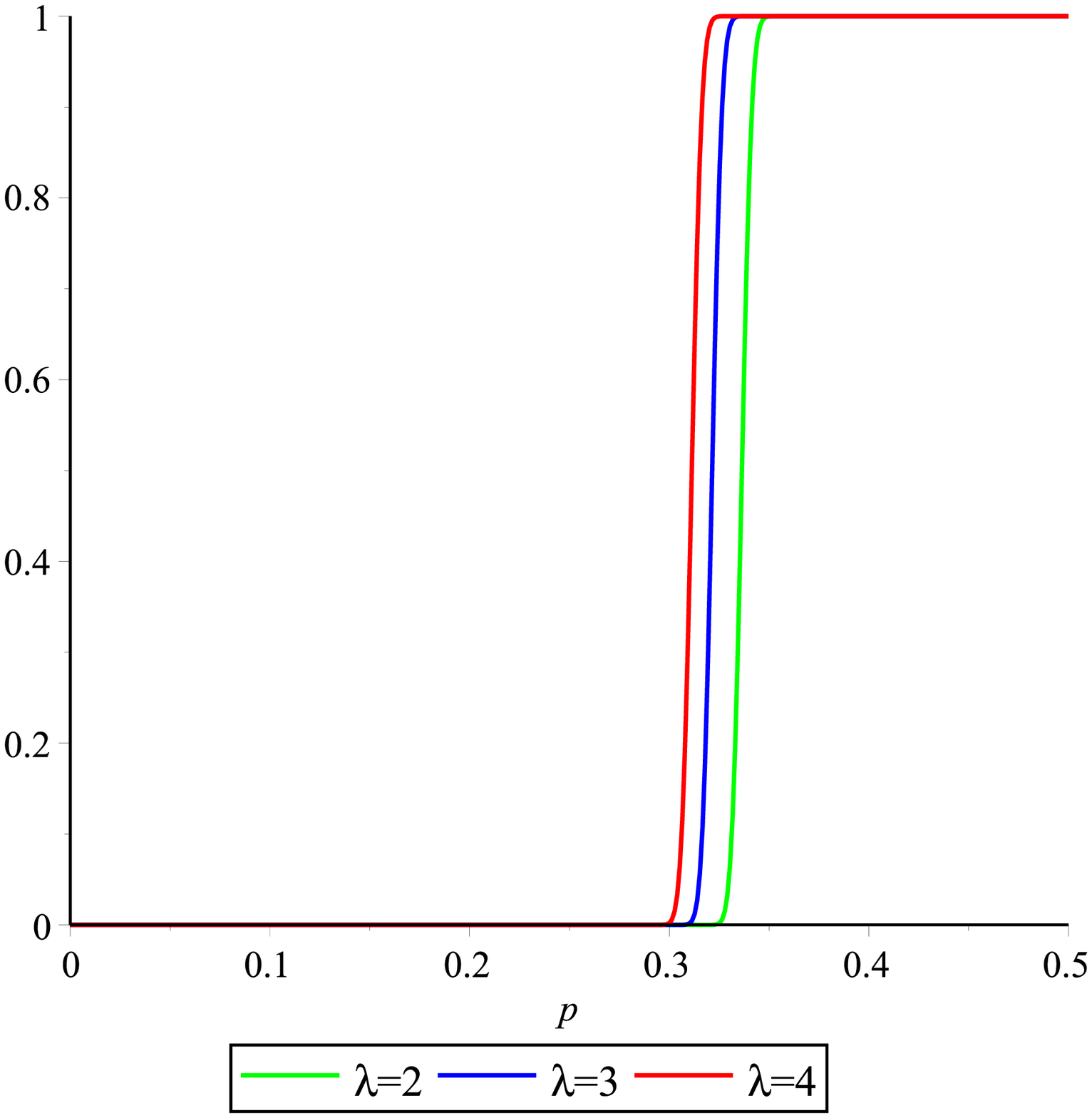}&
 		\includegraphics[width=0.5\textwidth]{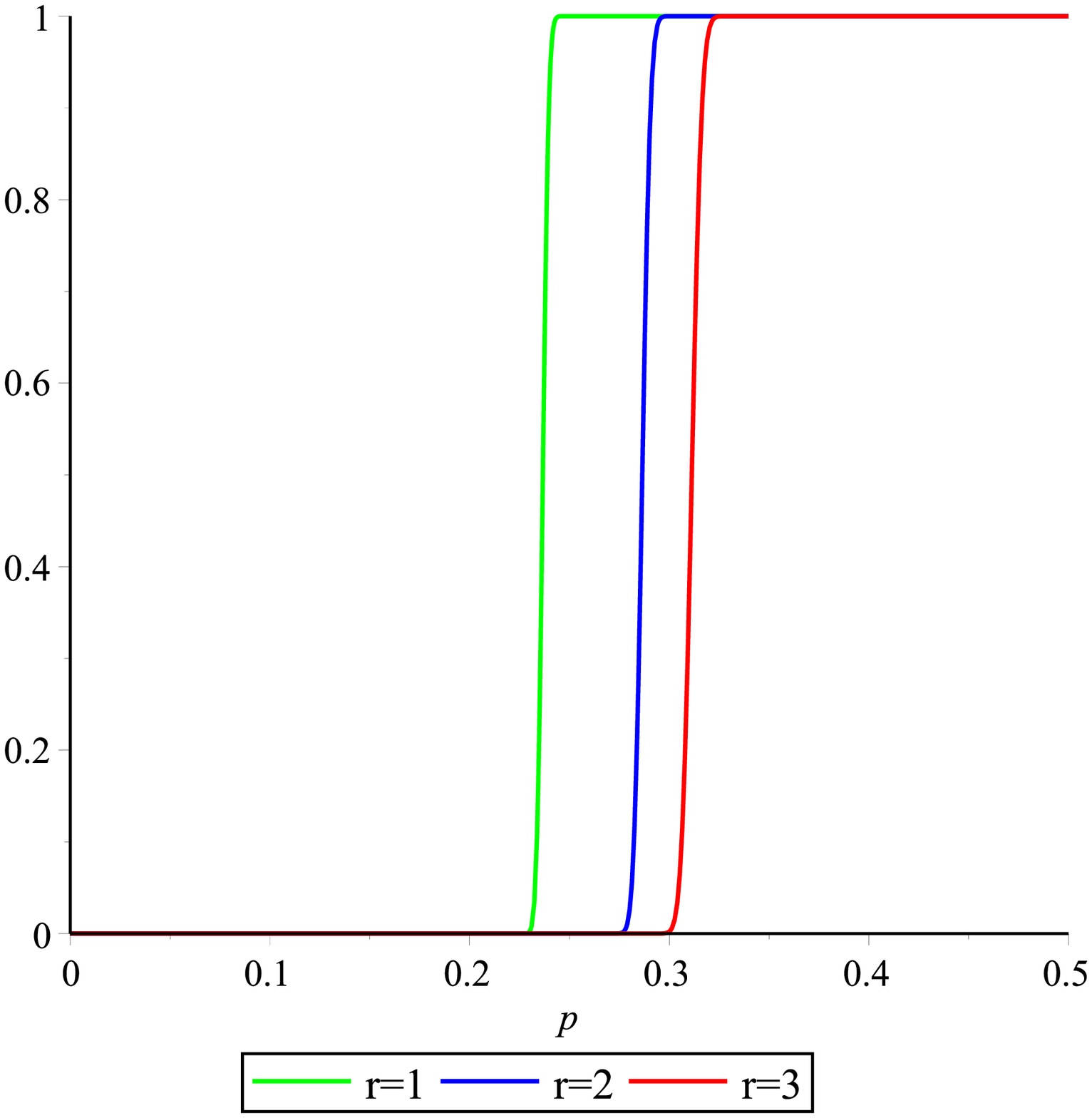}\\
 	\end{tabular}
 	\caption{\small LHS: $\frac{s_{p,\lambda}^{[r]}(n)}{s_\lambda^{[r]}(n)}$ as a function of $p$ for ${\bf GC}$-sequences of length $n=400$ and $3$-canonical structures subject to	minimum arc-length restrictions $\lambda=2$(green), $3$(blue) and $4$(red), respectively.
          RHS:  $\frac{s_{p,\lambda}^{[r]}(n)}{s_\lambda^{[r]}(n)}$ as a function of $p$ for ${\bf GC}$-sequences of length $n=400$ and $r$-canonical structures with minimum arc-length $4$ subject to	minimum stack-length restrictions $r=1$(green), $2$(blue) and $3$(red), respectively.
 	} \label{F:threshold}
 \end{figure*}

\subsection{Watson-Crick base pairs}


We next consider RNA secondary structures realized by sequences composed by the four nucleotides
\textbf{A}, \textbf{U}, \textbf{C}, \textbf{G}, assuming Watson-Crick (\textbf{A-U}, \textbf{C-G}) base pairs.

For any  $n\in \mathbb{N}$ and $\mathbf{p}=(p_1,p_2,p_3,p_4)$, where $p_i\in ]0,1[$ and
$\sum_{i=1}^{4} p_i=1$, let $s_{\mathbf{p},\lambda}^{[r]}(n)$ denote the number of $r$-canonical RNA secondary structures
for which there exists at least one compatible \textbf{A}\textbf{U}\textbf{C}\textbf{G}-sequence containing
$\lfloor p_1 n\rfloor$,
$\lfloor p_2 n\rfloor$, $\lfloor p_3 n \rfloor$, and $n-\lfloor p_1 n\rfloor-\lfloor p_2 n\rfloor-
\lfloor p_3 n \rfloor$ nucleotides, respectively. Let $s_{\mathbf{p},\lambda}^{[r]}(n,l)$ be defined analogously.
Obviously, as observed for two letter sequences, as the sequence length approaches infinity, 
\[
\Big(\frac{\lfloor p_1 n\rfloor}{n},\frac{\lfloor p_2 n\rfloor}{n},\frac{\lfloor p_3 n \rfloor}{n},
\frac{n-(\lfloor p_1 n\rfloor+\lfloor p_2 n\rfloor+\lfloor p_3 n \rfloor)}{n}\Big)= (p_1,p_2,p_3,p_4)+O(n^{-1}).
\]

\begin{proposition}\label{P:wc}
	For any $n,r,l,\lambda \in \mathbb{N}$ and $\mathbf{p}$, let
	$$
	l_0=\min(\lfloor p_1 n\rfloor ,\lfloor p_2n\rfloor)+
	\min(\lfloor p_3 n\rfloor ,n-(\lfloor p_1 n\rfloor+\lfloor p_2 n\rfloor+\lfloor p_3 n \rfloor)).
	$$
	Then we
	have 
	\begin{equation}\label{Eq:pp2}
	\begin{aligned} 
	s_{\mathbf{p},\lambda}^{[r]}(n,l)&= s_\lambda^{[r]}(n,l) \qquad \text{if } l\leq l_0 ,\\
	s_{\mathbf{p},\lambda}^{[r]}(n,l)&= 0 \qquad \text{ $\qquad$ otherwise},
	\end{aligned}
	\end{equation}
	and
	\begin{equation}\label{Eq:ppsum2}
	s_{\mathbf{p},\lambda}^{[r]}(n)= \sum_{l=0}^{l_0 } s_\lambda^{[r]}(n,l).
	\end{equation}	
\end{proposition}
\begin{proof}
	We prove that any structure having $l$ arcs is counted by $s_{\mathbf{p},\lambda}^{[r]}(n,l)$ if $l\le l_0$.
	Given a structure $S$ having $l$ arcs, each arc corresponds to an \textbf{A-U} pair or a \textbf{C-G} pair
	for any compatible sequence, while each unpaired vertex could be either one of the four bases.
	Suppose we have $l_1$ \textbf{A-U} and $l_2$ \textbf{C-G} base pairs. Then $l_1\le \lfloor p_1 n\rfloor$ and
	since $\lfloor p_2 n\rfloor$ is the number of \textbf{U} nucleotides, we have $l_1 \le \lfloor p_2 n\rfloor$.
	The case of \textbf{C-G} pairs is treated analogously. Thus if $l\le l_0$ $S$ is counted by
	$s_{\mathbf{p},\lambda}^{[r]}(n,l)$. We inspect that then for any $l'\le l$ any structure containing $l'$ arcs is
	counted by $s_{\mathbf{p},\lambda}^{[r]}(n,l)$.
	
	We thus arrive at the following integer linear programming problem:
	\begin{align*}
	\text{maximize: }\quad &l=l_1+l_2\\
	\text{subject to }\quad & l_1 \leq \lfloor p_1 n\rfloor,\, l_1 \leq
	\lfloor p_2 n\rfloor,\, l_2 \leq \lfloor p_3 n\rfloor,\, l_2 \leq
	n-(\lfloor p_1 n\rfloor+\lfloor p_2 n\rfloor+\lfloor p_3 n \rfloor) \\
	\text{and } \quad & l_1\geq  0,\, l_2\geq 0,\, l_1,l_2\in \mathbb{Z}.
	\end{align*}
	Here $l_0=\min(\lfloor p_1 n\rfloor ,\lfloor p_2 n\rfloor )+\min(\lfloor p_3 n\rfloor ,
	n-(\lfloor p_1 n\rfloor+\lfloor p_2 n\rfloor+\lfloor p_3 n \rfloor) )$ is a solution that
	is by construction unique.
	In case of $l> l_0$, $l$ is of the form
	$l=l_1+l_2$ such that
	$l_1 \geq \lfloor p_1 n\rfloor$ or  $l_1 \geq \lfloor p_2 n\rfloor$ or
	$l_2 \geq \lfloor p_3 n\rfloor$, or $l_2 \geq n-(\lfloor p_1 n\rfloor+\lfloor p_2 n
	\rfloor+\lfloor p_3 n \rfloor)$. This means however, that we have not enough of one type of the four
	bases to satisfy all $l$ base pairings of the structure, whence $s_{\mathbf{p},\lambda}^{[r]}(n,l)= 0$.
\end{proof}

\begin{theorem}\label{T:wc}
	Given any $1\leq \lambda \leq 4$, $1\leq r \leq 3$, any $n\in\mathbb{N}$ and $\mathbf{p}$, let
	$p_0=\min(p_1 ,p_2)+\min(p_3 ,p_4 )$. 
	Then we have
	\begin{equation}\label{Eq:estimate2}
	\frac{s_{\mathbf{p},\lambda}^{[r]}(n)}{s_\lambda^{[r]}(n)}= \frac{1}{\sqrt{2\pi}}
	\int_{-\infty}^{\frac{(p_0-\mu_\lambda^{[r]})\sqrt{n}}{\sigma_\lambda^{[r]}}}\,e^{-\frac{1}{2}t^2} \mathrm{d} t +
	O(n^{-\frac{1}{2}}),
	\end{equation}
	where $\mu_\lambda^{[r]},\sigma_\lambda^{[r]}$ are the mean and  standard deviation of $\mathbb{Y}_{\lambda,n}^{[r]}$.

        Equivalently, for $p_0\in]\mu_\lambda^{[r]},\frac{1}{2}]$,  a random structure a.a.s. has a compatible sequence with nucleotide ratio $p$. Conversely, in case of $p_0\in]0,\mu_\lambda^{[r]}[$, a.a.s. no random structure has a compatible sequence.
\end{theorem}
\begin{proof} 
	Given $\mathbf{p}=(p_1,p_2,p_3,p_4)$, analogous to the proof of Theorem~\ref{T:ppclt}, we derive by
	eq.~(\ref{Eq:pp2})
	\begin{align*}
	\frac{s_{\mathbf{p},\lambda}^{[r]}(n)}{s_{\lambda}^{[r]}(n)}&= \sum_{l=0}^{ l_0 } \frac{s_{\lambda}^{[r]}(n,l)}{s_{\lambda}^{[r]}(n)} 
	= \mathbb{P}(\mathbb{Y}_{\lambda,n}^{[r]}\leq  l_0 )
	=\frac{1}{\sqrt{2\pi}}\int_{-\infty}^{\frac{l_0-\mu_{\lambda}^{[r]}n }{\sigma_{\lambda}^{[r]}\sqrt{n}}}\,e^{-\frac{1}{2}t^2} \mathrm{d} t + O(n^{-\frac{1}{2}})\\
	&=\frac{1}{\sqrt{2\pi}}\int_{-\infty}^{\frac{p_0n-\mu_{\lambda}^{[r]}n}{\sigma_{\lambda}^{[r]}\sqrt{n}}}\,e^{-\frac{1}{2}t^2} \mathrm{d} t + O(n^{-\frac{1}{2}}), 
	\end{align*}
	where the last equality is implied by the inspecting $|l_0- p_0 n |\le 2$ and 
	\begin{align*}
	  \Big|\int_{-\infty}^{\frac{l_0-\mu_{\lambda}^{[r]}n}{\sigma_{\lambda}^{[r]}\sqrt{n}}}\,e^{-\frac{1}{2}t^2} \mathrm{d} t -\int_{-\infty}^{\frac{ p_0 n  -
          \mu_{\lambda}^{[r]}n }{\sigma_{\lambda}^{[r]}\sqrt{n}}}\,e^{-\frac{1}{2}t^2} \mathrm{d} t \Big|& =\Big|
          \int_{\frac{ p_0 n  -\mu_{\lambda}^{[r]}n }{\sigma_{\lambda}^{[r]}\sqrt{n}}}^{\frac{l_0-\mu_{\lambda}^{[r]}n}{\sigma_{\lambda}^{[r]}\sqrt{n}}}\,
          e^{-\frac{1}{2}t^2} \mathrm{d} t\Big|\le \frac{2}{\sigma_{\lambda}^{[r]}\sqrt{n}}\cdot e^{-\frac{1}{2}t^2}\Big|_{t=0}
	= O(n^{-\frac{1}{2}}),
	\end{align*}
	whence eq.~(\ref{Eq:estimate2}) and the theorem is proved.
\end{proof}



\subsection{Watson-Crick and wobble base pairs}


We  consider RNA secondary structures over four letter sequences, having Watson-Crick
(\textbf{A-U}, \textbf{C-G}) and wobble (\textbf{U-G}) base pairs.

For any  $n\in \mathbb{N}$ and $\mathbf{p}=(p_1,p_2,p_3,p_4)$, where $p_i\in ]0,1[$ and
$\sum_{i=1}^{4} p_i=1$, let $ \bar{s}_{\mathbf{p},\lambda}^{[r]}(n)$ denote the number of RNA secondary structures
for which there exists at least one compatible \textbf{AUCG}-sequence containing
$\lfloor p_1 n\rfloor$, $\lfloor p_2 n\rfloor$, $\lfloor p_3 n \rfloor$, and
$n-\lfloor p_1 n\rfloor-\lfloor p_2 n\rfloor-\lfloor p_3 n \rfloor$ nucleotides, respectively.
Let $\bar{s}_{\mathbf{p},\lambda}^{[r]}(n,l)$ be this quantity, filtered by the number of arcs.

\begin{proposition}
	For any $n,r,l,\lambda \in \mathbb{N}$ and $\mathbf{p}$ , let
	\begin{eqnarray*}
		\bar{l}_0 & = & \min\{\min\{\lfloor p_1 n\rfloor ,\lfloor p_2 n\rfloor \}+
		n-(\lfloor p_1 n\rfloor+\lfloor p_2 n\rfloor+\lfloor p_3 n \rfloor),\\
		&&  \qquad\   \min\{\lfloor p_3 n\rfloor  ,n-(\lfloor p_1 n\rfloor+\lfloor
                p_2 n\rfloor+\lfloor p_3 n \rfloor) \}
		+\lfloor p_2 n\rfloor \}.
	\end{eqnarray*}
	Then we have 
	\begin{equation}\label{Eq:pp3}
	\begin{aligned} 
	\bar{s}_{\mathbf{p},\lambda}^{[r]}(n,l)&= s_\lambda^{[r]}(n,l) \qquad \text{if } l\leq \bar{l}_0,\\
	\bar{s}_{\mathbf{p},\lambda}^{[r]}(n,l)&= 0 \qquad\quad \text{otherwise } ,
	\end{aligned}
	\end{equation}
	and
	\begin{equation}\label{Eq:ppsum3}
	\bar{s}_{\mathbf{p},\lambda}^{[r]}(n)= \sum_{l=0}^{\bar{l}_0 } s_\lambda^{[r]}(n,l).
	\end{equation}	
\end{proposition}
\begin{proof}
	We first prove that any structures having $l$ arcs are counted by
	$\bar{s}_{\mathbf{p},\lambda}^{[r]}(n,l)$ if $l\le \bar{l}_0$.

	Suppose we have $l_1$ \textbf{A-U} pairs, $l_2$ \textbf{C-G}
	pairs and $l_3$ \textbf{U-G} pairs. In order to be counted by $\bar{s}_{\mathbf{p},\lambda}^{[r]}(n,l)$, the number of
	\textbf{U} is at least $l_1+l_3$, i.e.~we have, $l_1+l_3\le \lfloor p_2 n\rfloor$ and since we have
	$\lfloor p_1 n\rfloor$ \textbf{A} nucleotides $l_1 \le \lfloor p_1 n\rfloor$.
	The situation for \textbf{C} and \textbf{G} is analogous.
	As long as $l_1$, $l_2$ and $l_3$ satisfy these conditions, we can always construct a compatible
	$\mathbf{p}$-sequence for a structure $S$ having $l=l_1+l_2+l_3$ arcs.
	Clearly, this holds for any $l'<l$, whence a structure having $l$ arcs is counted by
	$\bar{s}_{\mathbf{p},\lambda}^{[r]}(n,l)$ if the number of arcs  is not greater than the maximum of
	$l_1+l_2+l_3$ subject to the constraints of the following integer linear programming problem:
	\begin{align*}
	\text{maximize: }\quad &l=l_1+l_2+l_3\\
	\text{subject to }\quad & l_1 \leq \lfloor p_1 n\rfloor,\, l_1+l_3 \leq \lfloor p_2 n\rfloor,\, l_2 \leq \lfloor p_3 n\rfloor,\, l_2+l_3
        \leq n-(\lfloor p_1 n\rfloor+\lfloor p_2 n\rfloor+\lfloor p_3 n \rfloor)\\
	\text{and }\quad & l_1,l_2,l_3 \geq 0,\,l_1,l_2,l_3 \in \mathbb{Z}.
	\end{align*}
	Its unique solution is
	\begin{eqnarray*}
		\bar{l}_0 & = & \min\{\min\{\lfloor p_1 n\rfloor ,\lfloor p_2 n\rfloor \}+
		n-(\lfloor p_1 n\rfloor+\lfloor p_2 n\rfloor+\lfloor p_3 n \rfloor),\\
		&&  \qquad\   \min\{\lfloor p_3 n\rfloor  ,n-(\lfloor p_1 n\rfloor+\lfloor p_2 n\rfloor+\lfloor p_3 n \rfloor) \}
		+\lfloor p_2 n\rfloor \}
	\end{eqnarray*}
	and in case of $l> \bar{l}_0$, we observe as in Proposition~\ref{P:wc}, that we have not enough of at least one of
	the four nucleotide types, whence $\bar{s}_{\mathbf{p},\lambda}^{[r]}(n,l)= 0$ and the proposition follows.
\end{proof}

As a result we derive
\begin{theorem}\label{T:wcw}
	For $1\leq \lambda \leq 4$, $1\leq r \leq 3$, $n\in\mathbb{N}$ and $\mathbf{p}$, let
	$\bar{p}_0=\min(\min(p_1 ,p_2 )+p_4,\min(p_3 ,p_4 )+p_2)$. Then we have
	\begin{equation}\label{Eq:estimate3}
	\frac{\bar{s}_{\mathbf{p},\lambda}^{[r]}(n)}{s_\lambda^{[r]}(n)}=
	\frac{1}{\sqrt{2\pi}}\int_{-\infty}^{\frac{(\bar{p}_0-\mu_\lambda^{[r]})\sqrt{n}}{\sigma_\lambda^{[r]}}}\,
	e^{-\frac{1}{2}t^2} \mathrm{d} t + O(n^{-\frac{1}{2}}),
	\end{equation}
	where $\mu_\lambda^{[r]},\sigma_\lambda^{[r]}$ are the mean and  standard deviation of $\mathbb{Y}_{\lambda,n}^{[r]}$.

        Equivalently, 
	for $\bar{p}_0\in]\mu_\lambda^{[r]},\frac{1}{2}]$,  a random structure a.a.s. has a compatible
        sequence with nucleotide ratio $p$. Conversely, in case of $\bar{p}_0\in]0,\mu_\lambda^{[r]}[$,
            a.a.s. no random structure has a compatible sequence.	
\end{theorem}
\begin{proof} 
	The proof is analogous to that of Theorem~\ref{T:wc}, the only difference being that
	\[
	| \bar{l}_0 -\bar{p}_0 n| \leq 2.
	\] 
\end{proof}

In Fig. 1 from Supplementary Material, we display $\frac{\bar{s}_{\mathbf{p},\lambda}^{[r]}(n)}{s_\lambda^{[r]}(n)}$ as a function of
$p_1$ and $p_3$, employing the average ${\bf AUCG}$-ratios found in RNA databases. We contrast the cases of
Watson-Crick and Watson-Crick together with wobble base pairs.



\section{Discussion}\label{S:Dis}

%

In this section we discuss our findings in the context of nucleotide percentages observed in RNA databases \citep{Berman:00}
and provide some implications of our results. Using the data from the RCSB PDB database~\citep{Berman:00}, we list
in Table~\ref{Tab:AUCG} the average nucleotide ratios for several classes of RNA and the corresponding values of
$p_0$ and $\bar{p}_0$ (see Theorem~\ref{T:wc} and Theorem~\ref{T:wcw}).

The observed average \textbf{AUCG}-ratio of all RNA families is $(0.208,0.200,0.271,0.321)$ facilitating
according to Theorem~\ref{T:wc} and Theorem~\ref{T:wcw} the formation of asymptotically almost all structures,
see Fig. 1 from Supplementary Material.

\begin{table}
	\caption{The average nucleotide ratios and the values of $p_0, \bar{p}_0$ for different families of RNA structures.
	The theoretical value with whom  $p_0$ and $\bar{p}_0$ have to be compared is $\mu_4^{[3]}=0.3113$.} \label{Tab:AUCG}
	\begin{tabular}{cccccccc}
		\hline\noalign{\smallskip}
		\small     &\small\textbf{A}$$                    & \small\textbf{U}$$       & \small\textbf{C}$$        & \small\textbf{G}$$	&\small$\sigma^2$	&\small$p_0$	&\small$\bar{p}_0$           \\
		\noalign{\smallskip}\hline\noalign{\smallskip}
		\small\text{mRNA}  &\small$0.412$               & \small$0.265$ & \small $0.110$ &\small $0.213$	&\small$0.092$	&\small$0.375$	&\small$0.375$   \\
		\small\text{tRNA}  &\small$0.189$               & \small$0.201$ & \small $0.292$ &\small $0.318$	&\small$0.004$  	&\small$0.481$	&\small$0.493$   \\		
		\small$5$\text{S ribosomal RNA}  &\small$0.208$               & \small$0.165$ & \small $0.304$ &\small $0.323$	&\small$0.006$	&\small$0.470$	&\small$0.470$   \\
		\small$16$\text{S ribosomal RNA}  &\small$0.214$               & \small$0.165$ & \small $0.268$ &\small $0.353$	&\small$0.004$	&\small$0.433$	&\small$0.433$   \\
		\small$23$\text{S ribosomal RNA}  &\small$0.246$               & \small$0.180$ & \small $0.244$ &\small $0.330$	& \small$0.006$	&\small$0.424$	&\small$0.424$   \\
		\small\text{Bacteria}  &\small$0.238$               & \small$0.197$ & \small $0.247$ &\small $0.317$   	&\small$0.010$	&\small$0.444$	&\small$0.444$\\
		\small\text{Eukaryota}  &\small$0.232$               & \small$0.237$ & \small $0.238$ &\small $0.293$	&\small$0.021$	&\small$0.470$	&\small$0.475$   \\		\
		\small\text{Viruses}  &\small$0.182$               & \small$0.205$ & \small $0.298$ &\small $0.315$	&\small$0.015$	&\small$0.480$	&\small$0.497$   \\	
		\small\text{RNA}  &\small$0.208$               & \small$0.200$ & \small $0.271$ &\small $0.321$	&\small$0.052$	&\small$0.471$	&\small$0.471$   \\				
		\noalign{\smallskip}\hline
	\end{tabular}
\end{table}

Table~\ref{Tab:AUCG} shows, that the average ratio for mRNAs is significantly different from that of other families and
can thus be used as a discriminat for the mRNAs that can easily be obtained at the time of sequencing.
With a ratio of {\bf A} being $0.412$ and a $2\;:\;1$ ratio of {\bf G} to {\bf C} the mRNA family exhibits particular
nucleotide ratios making it it more difficult to form configurations having low free energies. For mRNA, $p_0$ and $\bar{p}_0$
are relatively close to the critical value $\mu_4^{[3]}=0.3113$ which motivates to look in more detail at mRNA structures,
in particular in the context of mRNA redesign with the objective to achieve a mfe-structure without affecting the induced amino
acid sequence~\citep{Gaspar:13}. \citet{Buratti:04} formulates the hypothesis that pre-mRNA molecules behaved essentially as a
linear structure, or ``tape'' so to speak. In particular, our results show that  pre-mRNAs with $p_0<\mu_4^{[3]}$ almost surely have a linear
structure.

We observe that for mRNAs whose critical values $p_0$ are greater than $\mu_4^{[3]}$, fold. For instance, mRNA HP210~\citep{Mahen:10}
with $p_0=\bar{p}_0=0.45$ folds into several distinct confirmations and interchanges them quickly (in vivo).

mRNA sequences with $p_0<\mu_4^{[3]}$, have a diminished capability of folding and typically interact with {other}
biopolymers: for instance, the two synthetic messenger RNAs
{\tiny
$$
5'\textrm{-} \textbf{UUUUUUUUUUUUUUUUUGGCAAGGAGGUUUUUUUUUUUUUUUUUUUUUUU}\textrm{-} 3',
$$}
having $p_0=0.08$, $\bar{p}_0=0.18$ and
{\tiny
$$
  5'\textrm{-} \textbf{GGCAAGGAGGUAAAAAUGAAAAAAAAA}\textrm{-} 3',
  $$}
having $p_0=\bar{p}_0=0.11$ are used to study the interaction of mRNA with the ribosome at different states of translation \citep{Yusupova:06}.
During the translation initiation these mRNAs combine with 16s ribosomal RNA, forming ribosome complexes, see Fig.~\ref{F:mRNA}.

 \begin{figure}
 	\centering
 	\includegraphics[width=1\textwidth]{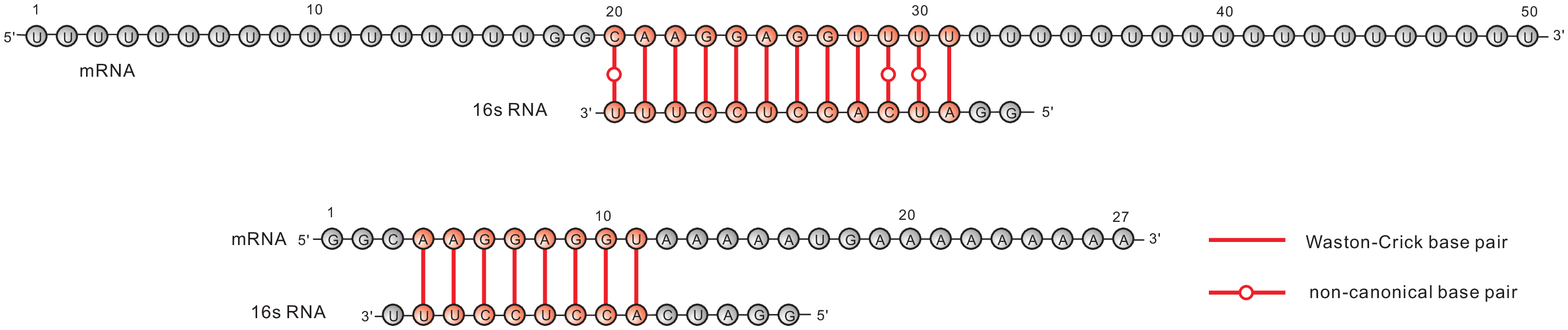}
 	\caption
 	{\small Interactions between mRNAs and 16s ribosomal RNA.
 	}
 	\label{F:mRNA}
 \end{figure}
A second instance of such interactions is reported in \citet{Rozov:15}. The messenger RNA
 {\tiny
   $$
   5'\textrm{-} \textbf{GGCAAGGAGGUAAAAAUGAACAAAAAAAAA}\textrm{-} 3',
   $$}
   ($p_0=\bar{p}_0=0.133$), pairs, as a result of a {\bf G-U} mismatch, with an incorrect tRNA and as a result produces an incorrect
 amino acid.

 Of course nucleotide ratios are a rather coarse criterion: within sequences of a fixed nucleotide ratio structural conformations
 depend on additional factors: \citet{Seffens:99} showed that mRNAs from databases have lower mfe than random RNA sequences of the
 same nucleotide ratio. Since mfe critically depends on the stacking of base pairs and loops, \citet{Workman:99} was able to
 show that random RNA sequences, generated with the same dinucleotide ratio, have not significantly higher mfe values.
 Accordingly, the dinucleotide ratio is the adaequate equivalence concept when studying mfe configurations.

 Irrespective of RNA family, if for an RNA sequence $p_0$ and $\bar{p}_0$ are smaller than $\mu_4^{[3]}=0.3113$, it is likely that
 it interacts with other biopolymers. In the following we list several instances of this phenomenon, see Fig.~\ref{F:interact},
\begin{enumerate}
\item RNAs found in the bluetongue virus~\citep{Diprose:02}, having three distinct segments of $412$, $276$, and $265$ bp,
  each of which comprised of two chains of \textbf{A}-sequences and \textbf{U}-sequences ($p_0=\bar{p}_0=0$) binding
  to each other by \textbf{A-U} base pairs;
\item the consensus sequence {\tiny $5'\textrm{-}\textbf{UACUAACACC}\textrm{-}3'$}  ($p_0=\bar{p}_0=0.2$) of the precursor
  (pre)－mRNA intron, interacting with the short region {\tiny $5'\textrm{-} \textbf{GGUGUAGUA}\textrm{-}3'$} ($p_0=0.222,\bar{p}_0=0.333$)
  of the U2 small nuclear (sn)RNA and forming a complementary helix of seven base pairs with a single, unpaired
  \textbf{A}-residue~\citep{Newby:02};
\item the secondary structure of the minimal, hinged hairpin ribozyme, formed by the interaction between the interdomain linker
      strand
      {\tiny
        $$
        5'\textrm{-}\textbf{CGGUGAGAAGGGXGGCAGAGAAACACACGA}\textrm{-}3',
        $$}
      ($p_0=\bar{p}_0=0.2$) and two other strands~\citep{Macelrevey:08}.
\end{enumerate}
We display the above three examples in Fig.~\ref{F:interact}.

  \begin{figure*}
  	\begin{tabular}{ccc}
 		\includegraphics[width=0.30\textwidth]{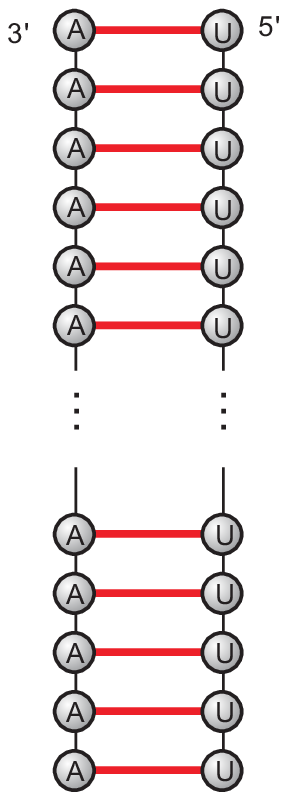}&
  		\includegraphics[width=0.25\textwidth]{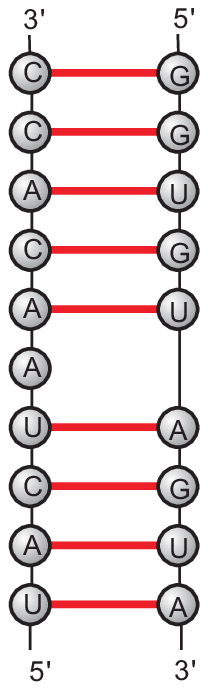}&
  		\includegraphics[width=0.35\textwidth]{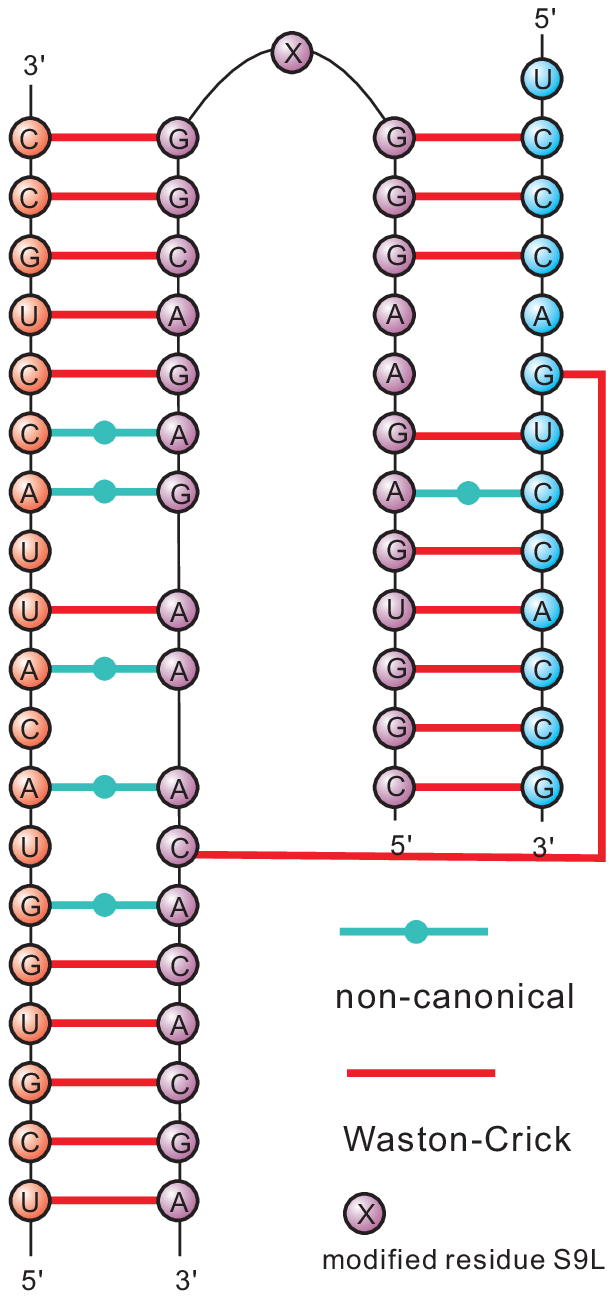}\\
  	\end{tabular}
  	\caption{LHS: the double-stranded structure~\citep{Diprose:02} comprised of an \textbf{A}-sequence and an \textbf{U}-sequence.
          M: the complementary helix~\citep{Newby:02} formed by the (pre)mRNA intron and U2 small nuclear (sn)RNA.
          RHS: the minimal, hinged hairpin ribozyme~\citep{Macelrevey:08} consisting of three interacting strands.
  	} \label{F:interact}
  \end{figure*}

The mathematical analysis presented here is currently extended to topological RNA structures, i.e.~RNA structures having pseudoknots.
These RNA structures are filtered by the topological genus of their corresponding fatgraphs. In this framework the concept of diagrams
is enriched by passing from graphs the fatgraphs. This is achieved by ordering the edges around a vertex and the secondary structures
discussed here are just topological RNA structures of genus zero.

\begin{center}
	{\bf ACKNOWLEDGMENTS}
\end{center}

We would like to thank Fenix W.D.~Huang for discussions and for providing the data on the
tRNA structures.

\begin{center}
	{\bf AUTHOR DISCLOSURE STATEMENT}
\end{center}

The authors declare that no competing financial interests exist.


\bibliographystyle{plainnat}      


\end{document}